\documentclass[12pt]{huber_article}

\usepackage{amsmath,amssymb,amsthm}
\usepackage{dsfont}
\usepackage{booktabs}
\usepackage{comment}
\usepackage{tikz}
\usepackage[round]{natbib}

\usetikzlibrary{shapes,positioning}

\tikzset{
block/.style={rectangle,draw,rounded corners}
}

\newcommand{\gk}{{\arabic{line})}\stepcounter{line}}

\newcounter{line}

\newcommand{\mean}{\mathbb{E}}

\newcommand{\prob}{\mathbb{P}}

\newcommand{\unifdist}{\textsf{Unif}}

\newcommand{\berndist}{\textsf{Bern}}

\newcommand{\expdist}{\textsf{Exp}}

\newcommand{\ind}{\ensuremath{\mathds{1}}} 

\begin{document}

\newtheorem{theorem}{Theorem}{
\newtheorem{lemma}{Lemma}
\newtheorem{fact}{Fact}
\theoremstyle{definition}
\newtheorem{definition}{Definition}

\title{Optimal linear Bernoulli factories for small mean problems\thanks{The final publication is available at Springer via http://dx.doi.org/10.1007/s11009-016-9518-3.}}

\author{Mark Huber \\ {\tt mhuber@cmc.edu}}

\maketitle

\begin{abstract}
Suppose a coin with unknown probability $p$ of heads can be flipped as often 
as desired.  
A Bernoulli factory for a function $f$ is 
an algorithm that uses flips of the coin together with auxiliary randomness
to flip a single coin with probability $f(p)$ of heads.  Applications
include perfect sampling from the stationary distribution of 
certain regenerative processes.  When $f$ is analytic, the problem can be 
reduced to a Bernoulli factory of the form $f(p) = Cp$ for constant $C$.
Presented here is a new algorithm that for small values of $Cp$, requires
roughly only $C$ coin flips.
From information theoretic 
considerations, this is also conjectured to be (to first
order) the minimum number of flips needed by any such algorithm.

For large values of $Cp$, 
the new algorithm can also be used to build a new Bernoulli
factory that uses only 80\% of the 
expected coin flips of the older method.  In addition, the new method
also applies 
to the more general problem of a linear 
multivariate Bernoulli factory, where there
are $k$ coins, the $k$th coin has unknown probability $p_k$ of heads, and the
goal is to simulate a coin flip with probability 
$C_1 p_1 + \cdots + C_k p_k$ of heads. 
\end{abstract}


\noindent {\bf MSC Classes: } 65C50, 68Q17

\newpage

\section{Introduction}

The notion of a Bernoulli factory was introduced in
\citet{asmussengt1992}
in the context of generating
samples exactly from the stationary distribution of a regenerative
Markov process.  A Bernoulli factory works as follows.  
Suppose we have the ability to draw
independent identically distributed (iid) Bernoulli random variables, each of
which is 1 with probability $p$ and 0 with probability $1 - p$ (write
$X \sim \berndist(p)$.)  Then given a function $f$, the goal is to use
a random number of draws from $X$ to build a new random variable which
is also Bernoulli, but with chance $f(p)$ of being 1 for a specified function
$f$.  In~\cite{asmussengt1992}, the needed function was a linear function,
namely a constant
times $p$.  This simple case generalizes:  
in~\cite{nacup2005} it was shown that
the ability to draw from $f(p) = 2p$ could be used to build a Bernoulli
factory for any analytic $f$ that was bounded away from 1.

The focus here is on building a nearly optimal linear 
Bernoulli factory where $Cp$ is known to be small.  This is nearly 
optimal in the sense that it uses (to first order) only $C$ flips of
the coin, and there is strong evidence to indicate that at least
$C$ flips are necessary.  As in~\cite{hubertoappearb}, 
a Bernoulli factory can be defined as follows.

\begin{definition}
Given $p^* \in (0,1]$ and a function $f:[0,p^*] \rightarrow [0,1]$, 
let $\cal A$ be a computable function
that takes as input a number $u \in [0,1]$ together
with a sequence of values in $\{0,1\}$, 
and returns an output in $\{0,1\}$.  For any $p \in [0,p^*]$, 
$X_1,X_2,\ldots$ iid $\berndist(p)$, and $U \sim \unifdist([0,1])$,
let $T$ be the infimum of times $t$ such that the value of 
${\cal A}(U,X_1,X_2,\ldots)$ only depends on the values of $X_1,\ldots,X_t$.
If the following holds, then call $\cal A$ a {\em Bernoulli factory}.
\begin{enumerate}
\item{$T$ is a stopping time with respect to the natural filtration
  that is finite with probability 1.}
\item{${\cal A}(U,X_1,X_2,\ldots) \sim \berndist(f(p))$.}
\end{enumerate}
Call $T$ the {\em running time} of the Bernoulli factory.
\end{definition}

Colloquially, a draw $X \sim \berndist(p)$ will be refereed to as 
a coin flip, or 
more specifically, a $p$-coin flip.
The result $X = 1$ corresponds to heads on the coin, while $X = 0$ indicates
tails.  So a Bernoulli factory attempts to flip a coin
with $f(p)$ chance of heads, by using a random number of
coin flips from the original coin together with some auxiliary randomness.

\cite{asmussengt1992} introduced 
Bernoulli factories for an application in perfect simulation, 
but did not show that they exist. \cite{keaneo1994} constructed the first 
general Bernoulli factories, showing
that such a factory with finite running time existed if and only if 
$f(p)$ was continuous over $[0,p^*]$ for some $p^* \in (0,1]$, and 
either it holds that $f(p)$ is identically 0 or 1, or that both 
$f(p)$ and $1 - f(p)$ are polynomially bounded away from 0 and 1 over
the allowable range of $p$.

Their strategy for building a Bernoulli factory was to construct 
Bernstein polynomials that approximated the function $f(p)$ as 
closely as possible.  Bernstein polynomials are linear combinations of
functions of the form $p^k(1 - p)^{n - k}$ where both $n$ and $k \leq n$
are nonnegative integers.  Such polynomials can be created from the coin
by flipping it $n$ times and seeing if exactly $k$ heads and $n - k$ tails
result.  
Keane and O'Brien  could show that the running time $T$ was finite with
probability 1 for their algorithm, but not much more.  
In particular, they could not show any bounds on the average running
time, or even that it was finite.

\cite{nacup2005} 
developed this approach further, and showed that Bernstein
polynomials could be constructed tightly enough that the running time
would have a finite expectation.  In addition, they showed that the 
tail of the distribution of their running time declined
exponentially.  Moreover, their work contained a proof that 
$f(p) = 2p$ is in a sense the most important function, since it can
be used to construct a Bernoulli factory for any function that 
is both real analytic over $[0,1]$ and bounded away from 1.

However, their approach was not a practical algorithm.
While the number of coin flips had finite expectation, 
the amount of memory and time needed to compute the function ${\cal A}$
grew exponentially with the number of flips.  Work of 
\cite{latuszynskikpr2011} solved this issue, and 
gave the first 
practical implementation of the Nacu and Peres approach.  Their
approach created a pair of reverse time processes, one a supermartingale,
the other a submartingale, that converged on the target $f(p)$.  The values
could be computed without the exponential overhead associated with the 
Nacu-Peres algorithm.

To bound $f(p) = Cp$ away from 1, they considered the function
$f(p) = \min\{Cp,1 - \epsilon\}$ so that the function was defined over
the entirety of $[0,1]$.  However, this was not strictly necessary,
as in the original application of \cite{asmussengt1992},
it was possible to easily insure that $f(p) \leq 1 - \epsilon$.  By 
not trying to sample from the function $f(p) = \min\{Cp,1-\epsilon\}$ for
all values of $p$, but only for those with $Cp \leq 1 - \epsilon$,
a new approach became possible.

This new approach was developed in \cite{hubertoappearb}, 
and was the first that did not begin with Bernstein polynomial approximations.
Instead, this approach used flips of the coin to alter the problem in ways that
insured that the final output had the correct distribution.  For instance,
suppose the goal was to generate a coin with probability of heads $2p$.
Then flip the original coin once.  If the coin is heads, the the output
is heads.  Otherwise, it is necessary to flip a $p/(1 - p)$-coin.

That way, the chance the final output is heads is 
$p(1) + (1 - p)\cdot p/(1 - p) = 2p$.  By advancing carefully in this
manner, it was shown 
how to build a Bernoulli factory
such that for $Cp \leq 1 - \epsilon$ where $\epsilon$ is a known constant,
\[
\mean[T] \leq 9.5 C \epsilon^{-1}.
\]

Moreover, the same work showed that this running time 
is the best possible up to a constant.  Specifically, in~\cite{hubertoappearb}
it was shown that that any Bernoulli
factory (to first order) must use on average at least $0.04C \epsilon^{-1}$ 
coin flips.
It remains an open question what the 
best constants for the lower and upper bounds
are, although there is strong reason to believe (see Section~\ref{SEC:lower})
that on average at least
$C$ flips (to first order) are necessary to generate a $Cp$ coin.

The primary application of the Bernoulli factory, starting 
with Asmussen, Glynn, and Thorisson (1992), 
is to generate perfect samples from the 
stationary distribution of regenerative Markov chains.
In~\cite{leedl2014}, it was shown how to use the Bernoulli factory 
in~\cite{hubertoappearb} to generate coins where 
$Cp \leq 2/3$.  Under this condition, the older algorithm gave a bound
of 38 coin flips on the average number needed.

Without going into the details of their algorithm, by tripling the expected
running time of there algorithm, it is possible to ensure that $Cp \leq 2/9$.
Under these conditions, the expected number of flips for the new algorithm
is bounded above by 7.8 (see Theorem~\ref{THM:main}) giving an algorithm
that only requires 62\% as many work on average as the old one.

This work gives the following results.
\begin{enumerate}
 \item  For $Cp$ small, an algorithm will be given that uses only $C$ coin
        flips on average.
 \item
    For $Cp$ at most $1 - \epsilon$ for known $\epsilon$, 
    an algorithm will be given that uses 
    only $7.57C\epsilon^{-1}$ coin flips on average.
 \item
    The new algorithm can be extended from the function $Cp$ for 
    single variate coins to the multivariate coin problem where there
    are $k$ coins with unknown means $p_1,\ldots,p_k$.  Suppose
    the goal is to generate $\berndist(r)$ where 
    \begin{equation}
       \label{EQN:r}
       r(p_1,\ldots,p_k) = C_1 p_1 + \cdots + C_k p_k.
    \end{equation}  
    Then setting $C = C_1 + \cdots + C_k$, the algorithm in the multivariate
    case has running time equal to the single coin case.
\end{enumerate}

More precisely, the running time of the new algorithm is given as follows.

\begin{theorem}
\label{THM:main}
Suppose  it is known that $Cp \leq M$ for a constant $M < 1/2$.  
Then there exists an algorithm for producing a $Cp$-coin that uses
on average
at most
\[
\frac{C}{(1 - 2M)(1 + Cp)}  + Cp\cdot \left[C 
  \frac{15.2}{1 - 2M + Cp}\right]
\]
coin flips.
\end{theorem}
This theorem is shown in Section~\ref{SEC:smallr}.  The multivariate
version is similar, and is shown in Section~\ref{SEC:multivariate}.
\begin{theorem}
\label{THM:second}
Let $C = C_1 + \cdots + C_k$, and $r = C_1 p_1 + \cdots + C_k p_k$,
Suppose  it is known that $r \leq M$ for a constant $M < 1/2$.  
Then there exists an algorithm for producing an $r$-coin that uses
on average
at most
\[
\frac{C}{(1 - 2M)(1 + r)}  + Cp\cdot \left[C 
  \frac{15.2}{1 - 2M + r}\right]
\]
flips from among the $k$ coins.
\end{theorem}

The remainder of this paper is organized as follows.  
Section~\ref{SEC:small}
presents the algorithm for small $r$, and shows correctness and the 
bound on the running time.  Section~\ref{SEC:large}
gives the extension to the multivariate problem and for 
larger values of $r$, and also includes the proofs of correctness
and the bound on the running time.  Finally, Section~\ref{SEC:lower}
considers why $C$ flips is likely the best possible.

\section{The algorithm for small $Cp$}
\label{SEC:small}

Let $r = Cp$.  
The first piece of the algorithm is a method for drawing from the 
logistic Bernoulli factory 
\[
f(p) = \frac{r}{1 + r}
\]
that uses $T$ coins, where $\mean[T] = C/(1+r)$.

As usual, say that $X$ is exponential with rate $\lambda$ (write 
$X \sim \expdist(\lambda)$) if $X$ has density 
$f_X(s) = \exp(-\lambda s)\ind(s \geq 0)$.  Here $\ind(\cdot)$ is the 
indicator function that evaluates to 1 when the argument is true and 
0 when the argument is false.
The following basic facts about exponentials will prove useful.
\begin{fact}
\label{FCT:exp}
Let $X \sim \expdist(\lambda_1)$ and $Y \sim \expdist(\lambda_2)$ be
independent.  Then $\prob(X \leq Y) = \lambda_1/(\lambda_1 + \lambda_2)$.
\end{fact}

\begin{fact}[Memoryless]
If $X \sim \expdist(\lambda)$, then 
for $s > 0$, the conditional distribution of $X - s$ given 
$X > s$ is exponential with rate $\lambda$ as well.  That is,
$[X - s|X > s] \sim \expdist(\lambda)$.
\end{fact}

Exponentials can be employed to define a one dimensional Poisson point process.
\begin{definition}
Let $A_1,A_2,\ldots$ be independent and identically distributed (iid)
exponential random variables with rate $\lambda$.  
Then 
\[
P = \{A_1,A_1+A_2,A_1 + A_2 + A_3, \ldots\}
\]
forms a Poisson point process on $[0,\infty)$ of rate $\lambda$.
For $[a,b] \subset [0,\infty)$, $P \cap [a,b]$ is a Poisson point
process on $[a,b]$ of rate $\lambda$.
\end{definition}

Several well known facts about Poisson point processes are useful.
\begin{fact}
\label{FCT:converse}
The converse of the definition holds:  any Poisson point process 
$P \subset [0,\infty)$ of rate $\lambda$
with points $0 < P_1 < P_2 < \cdots$ has
$P_{1} \sim \expdist(\lambda)$ and $P_{i}-P_{i-1} \sim \expdist(\lambda)$,
and all these exponentials are independent.
\end{fact}

\begin{fact}
\label{FCT:thinning}
Let $P = \{P_1,P_2,\ldots\}$ be a Poisson point process.  Let
$B_1,B_2,\ldots$ be a sequence of iid $\berndist(p)$ random variables.
Then $P' = \{P_i:B_i = 1\}$ is a Poisson point process of rate $\lambda p$.
[The process $P'$ is called the {\em thinned} process.]
\end{fact}

\begin{fact}
\label{FCT:combine}
Let $P_1$ and $P_2$ be independent Poisson point processes of rate
$\lambda_1$ and $\lambda_2$ over $[0,\infty)$.  Then
$P_1 \cup P_2$ is a Poisson point process of rate $\lambda_1 + \lambda_2$
over $[0,\infty)$.
\end{fact}

\begin{fact}
\label{FCT:mean}
The expected number of points in a Poisson point process of rate
$\lambda$ over $[a,b]$ is Poisson distributed with mean $\lambda(b - a)$.
\end{fact}


These ideas can be used to build the logistic Bernoulli factory for 
$r/(1+r)$.

\begin{center}
\setcounter{line}{1}
\begin{tabular}{rl}
\toprule
\multicolumn{2}{l}{{\tt Logistic\_Bernoulli\_Factory} \quad {\em Input: } 
  $C$} \\
\midrule 
\gk & $X \leftarrow 0$, draw $A \leftarrow \expdist(1)$ \\
\gk & Draw $T \leftarrow \expdist(C)$ \\
\gk & While $X = 0$ and $T < A$ \\ 
\gk & \hspace*{1em} Draw $B \leftarrow \berndist(p)$ \\
\gk & \hspace*{1em} If $B = 1$ then $X = 1$, else 
                   $T \leftarrow T + \expdist(C)$ \\
\gk & Return $X$ \\
\bottomrule
\end{tabular}
\end{center}

Note that line 4 can be accomplished in constant time (with $\Theta(k)$
preprocessing time) using the Alias
method of~\cite{walker1974}.

\begin{lemma}
\label{LEM:LBFoutput}
The output of {\tt Logistic\_Bernoulli\_Factory} is a Bernoulli with 
mean $r/(1 + r)$.
\end{lemma}

\begin{proof}
Let $T_1,T_2,\ldots$ be the successive values of $T$ taken on in the
algorithm, and $B_1,B_2,\ldots$ the successive values of $B$.  
Since $T_{i+1} - T_i$ is 
$\expdist(C)$, the $\{T_i\}$ form a Poisson point process $P$ of rate $C$.
Let $P'$ be the points $T_i \in P$ with $B_i = 1$.  Then $P'$ is a 
point process with rate $Cp = r$.

Let $T'_1 = \min\{P'\}$.  
The while loop examines the $P'$ process, and returns 1 if 
$T'_1 < A$, and 0 otherwise.  
By Fact~\ref{FCT:converse}, $T'_1 \sim \expdist(r)$, and 
$A \sim \expdist(1)$, so $\prob(T'_1 < A) = r/(1 + r)$ by Fact~\ref{FCT:exp}.
\end{proof}

\begin{lemma}
\label{LEM:LBFruntime}
In one call to 
{\tt Logistic\_Bernoulli\_Factory}, the expected number of coin flips
needed is $C/(1 + r)$.
\end{lemma}

\begin{proof}
The Poisson process of rate $C$ combined with the process of rate 1
forms a Poisson point process of rate $C + 1$.  The chance that any
point of this process is from the thinned rate $Cp$ process combined with the 
rate 1 process is $(Cp + 1)/(C + 1)$.  Therefore, the number of points
generated in the rate $C + 1$ process has a geometric distribution
with mean $(C+1)/(Cp + 1)$.  Each of these points has a $C/(C+1)$ chance
of coming from the rate $C$ process initially, and so requires a coin flip.
Therefore, combining these effects gives an expected number of coin
flips of $[C/(C+1)][(C+1)/(Cp + 1)] = C/(r + 1)$.
\end{proof}

Now suppose that there is a known $M < 1/2$ such that 
$r \leq M.$
Let ${\tt BF}(C)$ denote the 
Bernoulli Factory
from~\cite{hubertoappearb} that flips a $Cp$ coin using on average 
$9.5 C(1 - M)^{-1}$ flips of the original coin.
Consider the following algorithm.  
\begin{center}
\setcounter{line}{1}
\begin{tabular}{rl}
\toprule
\multicolumn{2}{l}{{\tt Small\_r\_1D\_Bernoulli\_Factory} \quad {\em Input: } 
  $C,M$} \\
\midrule 
\gk & $\beta \leftarrow 1/(1-2M)$ \\
\gk & Draw $Y \leftarrow {\tt Logistic\_Bernoulli\_Factory}(\beta C)$ \\
\gk & Draw $B \leftarrow \berndist(1/\beta)$ \\
\gk & If $Y = 0$, then $X \leftarrow 0$ \\
\gk & Elseif $Y = 1$ and $B = 1$, then $X \leftarrow 1$ \\
\gk & Else $X \leftarrow {\tt BF}(\beta C/(\beta - 1))$ \\
\bottomrule
\end{tabular}
\end{center}

\begin{lemma}
\label{LEM:Smallruntime}
Algorithm {\tt Small\_r\_1D\_Bernoulli\_Factory} produces a Bernoulli
distributed output with mean $Cp \leq M < 1/2$, 
and requires at most (on average)
\[
\frac{C}{(1 - 2M)(1 + Cp)}  + Cp\cdot \left[19C 
  \frac{1}{1 - 2M + Cp}\right]
\]
coin flips to do so.
\end{lemma}

Note that for small $p$ and $M$, this running time is to first order
just $C$.

\begin{proof}
First show correctness.  Let $A_1$ be the event that $Y = 1$ and $B = 1$ in
the algorithm (in which case line 5 sets $X$ to be 1), 
and $A_2$ be the event that $Y = 1$, $B = 0$, and a call to 
{\tt BF$(C\beta/(\beta - 1))$} returns a 1 (in which case line 7 sets
$X$ to be 1).  These are disjoint events, and the output of the algorithm is
$X = \ind(A_1) + \ind(A_2)$.  Therefore,
\[
\prob(X = 1) = \prob(A_1) + \prob(A_2).
\]

The value of $Y$ is the call to {\tt Logistic\_Bernoulli\_Factory}$(\beta C)$,
and so $\prob(Y = 1) = \beta Cp/(1+\beta Cp)$.  For the Bernoulli $B$, 
$\prob(B = 1) = 1/\beta$.
Therefore $\prob(A) = \prob(Y=1)\prob(B=1) = Cp/(1+\beta Cp).$

The output of {\tt BF}$(C\beta/(\beta - 1))$ is 1 with probability 
equal to $\beta C p$, so
\[
\prob(A_2) = \frac{\beta C p}{1 + \beta Cp}(1 - 1/\beta)
  \frac{\beta C p}{\beta - 1}
 = Cp \frac{\beta Cp}{1 + \beta Cp}
\]
Therefore
\[
\prob(X = 1) = Cp \frac{1}{1 + \beta Cp} 
  + Cp \frac{\beta Cp}{1 + \beta Cp}
  = Cp
\]
as desired.

Now for the running time.  Lemma~\ref{LEM:LBFruntime} gives a 
running time of $\beta C/(1 + Cp)$ for line 1.  Line 5 is executed with
probability $(\beta - 1)Cp/(1+\beta Cp)$.  By the way $\beta$ was
chosen, $Cp\beta/(\beta - 1) \leq 1/2$.  Therefore, 
Theorem 1.1 of~\cite{hubertoappearb} gives that the call to 
${\tt BF}(C\beta/(\beta - 1))$ requires at most
$19 [C\beta/(\beta - 1)]$ flips.  Therefore, the total number of flips is
on average at most
\[
\frac{C}{(1 - 2M)(1 + Cp)} + Cp\cdot \left[\frac{19C}{1 - 2M + Cp}\right].
\]
\end{proof}

This is close to Theorem~\ref{THM:main}, but the constant of 19 in
the second term is larger.  To improve this algorithm, and eliminate
the need for the call to the old {\tt BF} algorithm, it is necessary
to consider what happens for larger $r$.

\section{Large $r$ algorithm}
\label{SEC:large}

In this section, the algorithm of the previous section is improved
to allow for all $r \in [0,1-\epsilon]$, where $\epsilon$ is arbitrarily
close to 0.  Along the way, the older $9.5C\epsilon^{-1}$ algorithm 
of~\cite{hubertoappearb} is improved to a $7.5C\epsilon^{-1}$ algorithm.

The first step is to build a random coin flip whose mean is slightly
larger than $r$.  If this coin is tails, return tails for $r$.  If the 
coin returns heads, heads will be returned with probability close to 1.
Otherwise, a new coin will need to be flipped.

\subsection{A coin flip with mean slightly larger than $r$}

Consider an asymmetric random walk on the 
integers $\Omega = \{0,1,\ldots,m\}$, where given the current state $X_t$, the 
next state is either $\max\{0,X_t - 1\}$, or $\min\{X_t + 1,m\}$.
The transition probabilities are 
\[
\prob(X_{t+1} = \min\{i + 1,m\}|X_t = i) = p_r, \ 
\prob(X_{t+1} = \max\{i - 1,0\}|X_t = i) = q_r,
\] 
where $p_r + q_r = 1$.
This is 
also called the {\em Gambler's Ruin} walk.

The following facts about this well known process will be helpful.

\begin{fact}
\label{FCT:hitting}
Suppose $p_r \neq q_r$ 
and $T = \inf\{t:X_t \in \{0,m\}\}$.  Then
\begin{align}
\prob(X_T = m) &= \frac{1 - (q_r/p_r)^{X_0}}{1 - (q_r/p_r)^m} \\
\mean[T] &= \frac{X_0}{q_r - p_r} - \frac{m}{q_r - p_r}\cdot \prob(X_T = m).
 \label{EQN:gamblersruinruntime}
\end{align}
\end{fact}

\begin{fact}
Suppose $X_0 = m$, $p_r < q_r$, and $T = \inf\{t:X_t = 0\}$.
Then $\mean[T] \leq m/(q_r - p_r)$.
\end{fact}


Consider the following Bernoulli factory that begins a Gambler's Ruin
walk starting at state 1, and returns heads if the state reaches 0 before
it reaches $m$.
\begin{center}
\setcounter{line}{1}
\begin{tabular}{rl}
\toprule
\multicolumn{2}{l}{{\tt A} \quad {\em Input: } 
  $m,C$} \\
\midrule 
\gk & $s \leftarrow 1$ \\
\gk & While $s \in \{1,2,\ldots,m-1\}$ \\
\gk & \hspace*{1em} 
  $B \leftarrow {\tt Logistic\_Bernoulli\_Factory}(C)$ \\ 
\gk & \hspace*{1em} $s \leftarrow s - 2B + 1$ \\
\gk & Return $\ind(s = 0)$ \\
\bottomrule
\end{tabular}
\end{center}

\begin{lemma}
\label{LEM:Acorrect}
The output of {\tt A} is a Bernoulli with mean 
$r(1 - r^{m-1})/(1 - r^m)$.  
\end{lemma}

\begin{proof}
{\tt Logistic\_Bernoulli\_Factory} outputs a Bernoulli that has mean
$r/(1 + r)$.  Hence $p_r = 1/(1+r)$, $q_r = r/(1 + r)$, and $q_r/p_r = r$.
From Fact~\ref{FCT:hitting}, in line 5 that makes
$\prob(I = 0) = 1 - (1 - r)/(1 - r^m) = (r - r^m)/(1 - r^m) = 
 r(1 - r^{m-1})/(1 - r^m)$.  
\end{proof}

\begin{lemma}
\label{LEM:Aruntime}
The expected number of coin flips used by {\tt A}
is at most $C(m - 1)$.
\end{lemma}

\begin{proof}
As in the last proof $p_r = 1/(1+r)$ and $q_r = r/(1+r)$.  So
\[
q_r - p_r = \frac{r}{1 + r} - \frac{1}{1 + r} = -\frac{1 - r}{1 + r},
\]
and $(q_r/p_r) = r$.
Using~\eqref{EQN:gamblersruinruntime}, 
\begin{align*}
\mean[T] &= \frac{1 + r}{1 - r} \cdot \left[m \frac{1 - r}{1 - r^m} - 1 \right]
 = (1+r)\left[\frac{m}{1 - r^m} - \frac{1}{1 - r}\right].
\end{align*}

Let $f(r) = m/(1 - r^m) - 1/(1 - r).$  Then it holds that
$f(r) < m - 1$ for all $r \in (0,1)$ and $m \geq 1$.   
Note that for all $r \in (0,1)$:
\begin{align*}
f(r) < m - 1 &\Leftrightarrow m(1 - r) - \frac{(1 - r^m)}{1-r} < 
   (m - 1)(1 - r^m) \\
 &\Leftrightarrow m - mr - 1 - r - \cdots - r^{m-1} < 
  m - 1 - r^m(m - 1) \\
 &\Leftrightarrow (m - 1)r^m < r + r^2 + \cdots + r^{m-1} + mr \\
 &\Leftrightarrow m - 1 < r^{1-m} + r^{2-m} + \cdots + r^{-1} + mr^{1-m}.
\end{align*}
Since $r \in (0,1)$, $r^{i-m} \geq 1$ for all $i \in \{1,\ldots,m-1\}$,
so the right hand side is strictly greater than the left hand side.
Note $f(0) = m - 1$, so for $r \in [0,1 - \epsilon]$, the 
function is at most $m - 1$.  

Each call to line 3 requires on average
$C/(1 + r)$ time by Lemma~\ref{LEM:LBFruntime}.  Together, 
the overall number of 
steps (on average) is at most
\[
(1 + r)(m - 1) \frac{C}{1 + r} = C(m - 1).
\]
\end{proof}

\subsection{After the flip}

Here is how {\tt A} can be useful.  Using $A$, it is possible to generate
a Bernoulli random variable that is 1 with probability 
\[
p_\beta = \beta r \frac{1 - (\beta r)^{m-1}}{1 - (\beta r)^{m}}
\]
for any constant $\beta > 1$.  By choosing $\beta$ large enough,
$p_\beta \geq r$.  Note that $p_\beta / \beta \leq r$.  So
$r \in [p_\beta/\beta,p_\beta]$.

So the algorithm works as follows.  First flip a $p_\beta$-coin.
If it is tails, then return tails for the $r$-coin as well.  If it
is heads, then flip a $(1/\beta)$-coin.  If that is heads as well, 
return heads for the $r$-coin.  Otherwise, flip a $p'$-coin, and return
the same value for the $r$-coin.

For this algorithm to work, $p'$ must satisfy:
\[
r = \frac{p_\beta}{\beta} + p' p_\beta(1 - (1/\beta)).
\]
Solving for $p'$ gives
\begin{align*}
p' &= \frac{1}{\beta - 1} \left[
  \frac{(\beta r)^{m-1}}
       {1 + (\beta r)^1 + \cdots + (\beta r)^{m-2}}
 \right],
\end{align*}
so the next step of the algorithm is figuring out how to generate
a $p'$-coin.

\subsection{Generating a $p'$-coin}

Fortunately, we do not have to actually generate a $p'$-coin for all
possible values of $\beta$, as 
we are allowed to choose the value of $\beta$ to use, as long as 
$p_\beta \geq r$ for our choice of $\beta$.  Let
\[
\beta = 1 + \frac{1}{m-1},
\]
so $(\beta - 1)^{-1} = m-1$.  Then
\begin{equation}
\label{EQN:pprime}
p' = \frac{(m - 1)(\beta r)^{m - 1}}{1 + (\beta r) + \cdots + (\beta r)^{m-2}}. 
\end{equation}
Note that $p' \leq 1$ which gives that $p_\beta \geq r$ for this choice of
$\beta$.

The algorithm for generating a $p'$-coin for $p'$ as in~\eqref{EQN:pprime} 
will be called {\tt B} here, and 
is shown graphically in 
Figure~\ref{FIG:pprime}.  Notice that if the first flip is heads and the 
second flip is tails, then our problem has changed to the same problem, but
with $m$ reduced to $m - 1$.

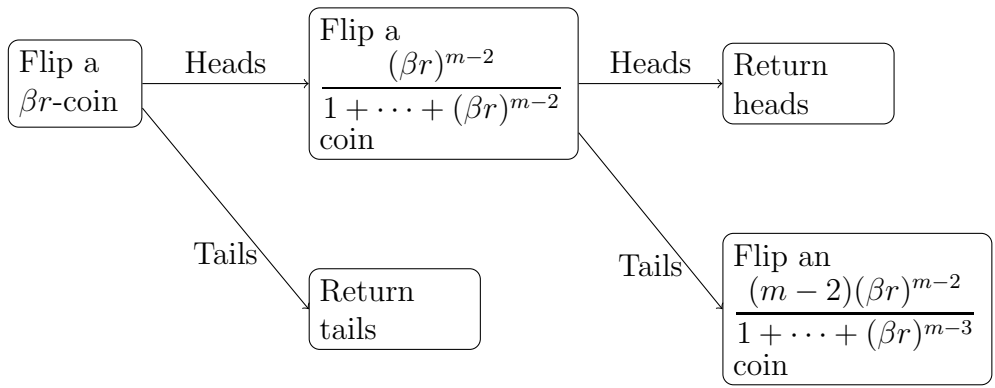
\begin{figure}[h!]
\begin{center}
\begin{tikzpicture}
\draw (0,0) node[text width=1.5cm,anchor=west,rectangle,draw,rounded corners] 
  (A) {Flip a ${\beta r}$-coin};
\draw (4,0) node[text width=3.3cm,anchor=west,rectangle,draw,rounded corners] 
  (B) {Flip a 
$\dfrac{(\beta r)^{m-2}}{1 + \cdots + (\beta r)^{m-2}}$ coin};
\draw (4,-3) node[text width=2cm,anchor=west,rectangle,draw,rounded corners] 
  (C) {Return tails};
\draw (9.5,0) node[text width=2cm,anchor=west,rectangle,draw,rounded corners] 
  (D) {Return heads};
\draw (9.5,-3) 
  node[text width=3.3cm,anchor=west,rectangle,draw,rounded corners] 
  (E) {Flip an 
  $\dfrac{(m -2)(\beta r)^{m-2}}{1 + \cdots + (\beta r)^{m - 3}}$ coin};
\draw[->] (A) -- node[above] {Heads} (B);
\draw[->] (A.340) -- node[below=0.3cm] {Tails} (C.180);
\draw[->] (B) -- node[above] {Heads} (D.180);
\draw[->] (B.340) -- node[below=0.3cm] {Tails} (E.180);
\end{tikzpicture}
\end{center}
\caption{A graphical illustration of Algorithm {\tt B}.}
\label{FIG:pprime}
\end{figure}

To utilize this procedure, it is necessary to be able to generate
a $(\beta r)^{m-2}/(1 + \cdots + (\beta r)^{m-2}$ coin.  Fortunately,
this can be accomplished fairly quickly using the Gambler's ruin chain
from earlier.

\begin{center}
\setcounter{line}{1}
\begin{tabular}{rl}
\toprule
\multicolumn{2}{l}{{\tt High\_Power\_Logistic\_BF}} \\
\multicolumn{2}{l}{{{\em Input: } 
  $m,\beta,C$} {\em Output: } 
  $X \sim \berndist((\beta r)^{m}/(1 + \cdots + (\beta r)^m))$} \\
\midrule 
\gk & $s \leftarrow 1$ \\ 
\gk & While $s \in \{1,\ldots,m\}$ \\
\gk & \hspace*{1em} Draw $B \leftarrow 
  {\tt Logistic\_Bernoulli\_Factory}(\beta C)$ \\
\gk & \hspace*{1em} $s \leftarrow s + 2B - 1$ \\
\gk & $X \leftarrow \ind(s = m + 1)$ \\
\bottomrule
\end{tabular}
\end{center}

\begin{lemma}
\label{LEM:HPLBFruntime}
The output of {\tt High\_Power\_Logistic\_BF} has
distribution $\berndist((\beta r)^m/(1 + \cdots + (\beta r)^m))$.  The
expected number of coin flips used is at most $\beta C/(1 - \beta r)$.
\end{lemma}

\begin{proof}
This is a Gambler's ruin where $p = \beta r/(1 + \beta r)$ 
and $q = 1/(1 + \beta r)$,
so $q/p = 1/(\beta r)$.  Hence from Fact~\ref{FCT:hitting},
\[
\prob(s = m + 1) = \frac{1 - (1/(\beta r))^1}{1 - (1/(\beta r))^{m+1}} = 
 \frac{(\beta r)^{m}(1 - \beta r)}{1 - (\beta r)^{m+1}} = 
  \frac{(\beta r)^{m}}{1 + \cdots + (\beta r)^{m}}. 
\]

Also from Fact~\ref{FCT:hitting}, note $q - p = (1+\beta r)/(1- \beta r)$, so
if $T$ is the number of times line
3 is called,
\begin{align*}
\mean[T] =
  \frac{1 + \beta r}{1 - \beta r}\left[1 - (m+1)\frac{(\beta r)^m(1-\beta r)}
  {1 - (\beta r)^{m+1}}\right] \leq \frac{1+\beta r}{1 - \beta r}.
\end{align*}

Each call to {\tt Logistic\_Bernoulli\_Factory} takes time
$\beta C/(1 + \beta r)$, so the overall number of coin flips 
(on average) is at most
$\beta C/(1 -\beta r)$. 
\end{proof}

In pseudocode, algorithm {\tt B} looks like this.
\begin{center}
\setcounter{line}{1}
\begin{tabular}{rl}
\toprule
\multicolumn{2}{l}{{\tt B} \quad {\em Input: } 
  $\epsilon,m,\beta,C$ \quad {\em Output: } $X$} \\
\midrule 
\gk & $X \leftarrow 0.5$ \\
\gk & While $X \notin \{0,1\}$ \\
\gk & \hspace*{1em} Draw $B_1 \leftarrow 
  {\tt Linear\_Bernoulli\_Factory}
  (1 - (1 - \epsilon)\beta,\beta\cdot C)$ \\
\gk & \hspace*{1em} If $B_1 = 0$ then $X \leftarrow 0$ \\
\gk & \hspace*{1em} Else \\
\gk & \hspace*{2em} $B_2 \leftarrow 
 {\tt High\_Power\_Logistic\_BF}(m-2,\beta,C)$ \\
\gk & \hspace*{2em} If $B_2 = 1$ then $X \leftarrow 1$ \\
\gk & \hspace*{2em} Else $m \leftarrow m - 1$ \\
\bottomrule
\end{tabular}
\end{center}

The most important thing to note here is that like many perfect simulation
algorithms, this method employs recursion.  We do not yet have an
algorithm for completing line 3!  However, this algorithm {\tt B} can
be used as a subroutine to create such an algorithm, and then this
subroutine will call the finished algorithm.

\begin{lemma}
\label{LEM:Bcorrect}
The output of {\tt B} has distribution
\[
\berndist((m - 1)(\beta r)^{m-1}/(1 + \cdots + (\beta r)^{m-2})).
\]
\end{lemma}

\begin{proof}
The proof is by induction.  When $m = 2$, if $B_1 = 1$ then 
$B_2 \sim \berndist(1)$, so $X = 1$ with probability $\beta r$ as desired.

Now suppose that the result holds for $m$, consider $m + 1$.  Then
\begin{align*}
\prob(X = 1) &= \beta r\left[\frac{(\beta r)^{m-2}}{1 + \cdots + (\beta r)^{m-2}}
 + \frac{1 + \cdots + (\beta r)^{m - 3}}{1 + \cdots + (\beta r)^{m - 2}}
  \cdot \frac{(m - 2)(\beta r)^{m-2}}{1 + \cdots + (\beta r)^{m - 3}} 
  \right] \\
 &= \frac{(m - 1)(\beta r)^{m - 1}}{1 + \cdots + (\beta r)^{m-2}},
\end{align*}
completing the induction.
\end{proof}

\subsection{The new linear Bernoulli Factory}

With these preliminaries in place, the overall algorithm is as follows.

\begin{center}
\setcounter{line}{1}
\begin{tabular}{rl}
\toprule
\multicolumn{2}{l}{{\tt Linear\_Bernoulli\_Factory} \quad {\em Input: } 
  $\epsilon,C$} \quad {\em Output: } $B$ \\
\midrule 
\gk & $m \leftarrow \lceil 4.5 \epsilon^{-1} \rceil + 1,$ 
      $\beta \leftarrow 1 + 1/(m-1)$ \\
\gk & $B_1 \leftarrow {\tt A}(m,\beta \cdot C)$ \\
\gk & If $B_1 = 1$ \\
\gk & \hspace*{1em} Draw $B_2 \leftarrow \berndist(1/\beta)$ \\
\gk & \hspace*{1em} If $B_2 = 1$ then $B \leftarrow 1$ \\
\gk & \hspace*{1em} Else \\
\gk & \hspace*{2em} Draw 
  $B \leftarrow {\tt B}(m,\beta,C)$ \\
\gk & Else $B \leftarrow 0$ \\
\bottomrule
\end{tabular}
\end{center}

Now consider the expected number of coin flips used by the algorithm.
As will become clear in the proofs of Lemmas~\ref{LEM:LBFexpected} 
and~\ref{LEM:correctruntime} below, making $m = \Theta(\epsilon^{-1})$
is the correct choice.  That leaves the choice of constant up to us, and
the constant of 4.5 from Line 1 was chosen to make the overall running
time as small as possible.

This algorithm calls {\tt A} and {\tt B}.  Line 2 of {\tt B} needs to 
draw a $\berndist(\beta r)$ random variable.  The best way to draw 
these random variables is to recursively call 
{\tt Linear\_Bernoulli\_Factory}.  In order to ensure
that this back in forth calling eventually comes to a halt 
with probability 1, it is easiest to bound the total expected number of 
calls to {\tt Linear\_Bernoulli\_Factory}.

\begin{lemma}
\label{LEM:LBFexpected}
The expected number of calls to {\tt Linear\_Bernoulli\_Factory} is at
most 1.4.
\end{lemma}

\begin{proof}
Let $m_1$ be the value of $m$ in the first call to 
{\tt Linear\_Bernoulli\_Factory}, and $\beta_1 = 1+ 1/(m_1 - 1).$
From this first call there is a chance of calling
{\tt B}, which in turn calls {\tt Linear\_Bernoulli\_Factory} with
$m_2$ and $\beta_2$.  Each of those second generation calls might call
a third generation, and so on.  To bound the expected number of calls
to {\tt Linear\_Bernoulli\_Factory} sum over all possible calls of the 
probability that that call is executed.  Let $N_i$ denote the 
number of $i$th generation calls.

The expected number of calls in the first generation is 1.
Consider a call in the second generation.  In order for that call to be
made, there must have been a call to {\tt B} from the first generation,
and all prior second generation calls from line 3 of {\tt B} must have
had $B_1 = 0$.  The number of times the while loop in {\tt B}
is executed is stochastically dominated by a geometric random variable
with mean $1/(1 - \beta_1 r)$.  Since
\begin{align}
1 - \beta_1 r &\geq 1 - (1 + 1/\lceil 4.5 \epsilon^{-1}\rceil)(1 - \epsilon) \\
 &= (7/9)\epsilon + (2/9)\epsilon^2,
\label{EQN:epsilonchange}
\end{align}
the number of calls made is bounded (in expectation) by 
$(9/7)\epsilon^{-1}$.  

But before {\tt B} is even called, first it must have held that
$B_1 = 1$ and $B_2 = 0$ in lines 2 and 4 of the first generation call
to {\tt Linear\_Bernoulli\_Factory}.

The probability that a call to 
{\tt B} is made is at most 
\[
(1 - 1/\beta_1)\beta r(1 - (\beta_1 r)^{m-1})/(1 - (\beta_1 r)^{m})
 \leq \frac{1}{m - 1}
\]

So the expected number of calls to {\tt Linear\_Bernoulli\_Factory} in
the second generation is
bounded by
\[
\mean[N_2|N_1] \leq N_1[1/(m_1-1)](9/7)\epsilon^{-1} \leq (2/7) N_1.
\]

This step forms the basis of an induction that gives
$\mean[N_i] \leq (2/7)^i N_1 = (2/7)^i$.
Therefore $\sum \mean[N_i] \leq 1/(1-2/7) = 1.4.$
\end{proof}

\begin{lemma}
\label{LEM:correct}
The output $B$ of {\tt Linear\_Bernoulli\_Factory} has 
$B \sim \berndist(r)$.
\end{lemma}

\begin{proof}
Line 2 of {\tt B} requires a draw $B_1 \leftarrow \berndist(\beta r)$.
Suppose that for the first $L$ times this line is called, the 
{\tt Linear\_Bernoulli\_Factory} is called to generate this random variable.
Then, from the $L+1$st time onwards, an oracle generates the random variable.

We show by strong induction that 
{\tt Linear\_Bernoulli\_Factory} generates from $\berndist(r)$ for any
finite $M$.  The base case when $M = 0$ operates as follows.  
Lemma~\ref{LEM:Bcorrect} immediately gives in this case that a call to 
{\tt B} returns a random variable with distribution
$\berndist((m-1)(\beta r)^{m-1}/(1 + \cdots + (\beta r)^{m-2}))$.
Lemma~\ref{LEM:Acorrect} gives that $B_1$ from line 2 has distribution
$\berndist((\beta r)^{m}/(1 + \cdots + (\beta r)^m))$.  Putting this
together gives
\begin{align*}
\prob(B = 1) &= (\beta r) \frac{1 - (\beta r)^{m-1}}{1 - (\beta r)^{m}}
 \left[\frac{1}{\beta} + 
    \left(1 - \frac{1}{\beta}\right)\frac{(m - 1)(\beta r)^{m-1}}{1 + \cdots +  
   (\beta r)^{m-2}}\right] \\
 &= (\beta r) \frac{1 - (\beta r)^{m-1}}{1 - (\beta r)^{m}}
 \left[\frac{1}{\beta} + 
    \frac{1}{\beta} \frac{(\beta r)^{m-1}}{1 + \cdots +  
   (\beta r)^{m-2}}\right] \\ 
 &= r \cdot \frac{(1 - (\beta r)^{m-1})}{1 - (\beta r)^m} \cdot 
   \frac{1 + \cdots + (\beta r)^{m-1}}{1 + \cdots + (\beta r)^{m-2}} \\
 &= r.
\end{align*}

This is the rare induction proof where the base case is just as hard as the 
induction step.  Suppose it holds for $L$, and consider what happens
for call 
limit $L + 1$.  Then the first call to {\tt Linear\_Bernoulli\_Factory}
might call {\tt B}, which might call {\tt Linear\_Bernoulli\_Factory}.  But
the first such call has used up one call, so only has $L + 1 - 1$
calls remaining, so by strong induction each 
returns the correct distribution.
Hence Lemma~\ref{LEM:Bcorrect}
holds, and the first call returns the correct distribution by the same
argument as the base case.

Let $N$ be the random number of calls to 
{\tt Linear\_Bernoulli\_Factory} needed by the algorithm.  Then let
$B$ be the output when $N$ is unbounded, and $B_M$ be the output when
a limit on calls equal to $L$ is in place.  Then 
\begin{align*}
\prob(B = 1) &= \prob(B = 1,N \leq L) + \prob(B = 1,N > L) \\
 &= \prob(B_L = 1,N \leq M) + \prob(B = 1,N > L) \\
 &= \prob(B_L = 1) - \prob(B_L = 1,N > L) + \prob(B = 1,N > L).
\end{align*}
Both $\prob(B_L,N > L)$ and $\prob(B = 1,N > L)$ are bounded above
by $\prob(N > L)$.  Since by the last lemma $\mean[N] \leq 1.4$, 
$\lim_{L \rightarrow\infty} \prob(N > L) = 0$.  The only way this can
hold for all $L$ is if $\prob(B = 1) = \prob(B_L = 1)$ for all $L$,
so $B$ has the correct distribution.
\end{proof}

\begin{lemma}
\label{LEM:correctruntime}
{\tt Linear\_Bernoulli\_Factory} uses on average at most
$7.67 C \epsilon^{-1}$ coin flips to generate $B \sim \berndist(r)$.
\end{lemma}

\begin{proof}
From the proof of 
Lemma~\ref{LEM:LBFexpected}, the expected number of calls
to the $i$th generation of {\tt Linear\_Bernoulli\_Factory} is 
bounded above by $(2/7)^i$.  

From~\eqref{EQN:epsilonchange}, 
at each successive generation of calls, $\epsilon$ is being multiplied by
a factor of at least $7/9$.  So an $i$th generation call to
{\tt Linear\_Bernoulli\_Factory} has an $m$ value of at most
$\lceil 4.5(9/7)^i \epsilon^{-1}\rceil + 1$, where $\epsilon$ was the input
for the 0th generation.  

Coin flips occur during the call to {\tt A}, and Lemma~\ref{LEM:Acorrect}
bounds the expected number of coin flips 
by $C \lceil 4.5(9/7)^i\epsilon^{-1} \rceil$.  So the expected total
flips coming from the $i$th generation of {\tt Linear\_Bernoulli\_Factory}
is at most $C[4.5(18/49)^i\epsilon^{-1} + (2/7)^i].$

Now look at the flips coming from an $i$th generation call to ${\tt B}$.  
This generation is only called from an $i$th generation call to 
{\tt Linear\_Bernoulli\_Factory}, of which there the expected number is at
most 
$(2/7)^i$.  
The call to {\tt B} occurs with probability at most $(\beta - 1)r$,
so at most $r/m$.  The while loop inside {\tt B} is run (on average) at
most $\epsilon^{-1}$ times, each of which could make a call to 
{\tt High\_Power\_Logistic\_BF}.  By Lemma~\ref{LEM:HPLBFruntime} this
requires at most $\beta C \epsilon^{-1}$ coin flips.  So the total 
number of coin flips from an $i$th generation call to ${\tt B}$ is at most
\[
(2/7)^i[r/m][\beta C \epsilon^{-1}(9/7)^{i+1}]\epsilon^{-1} \leq 
  (9/7)4.5^{-1}(18/49)^i C \epsilon^{-1}.
\]

Summing over these flips and the ones from {\tt Linear\_Bernoulli\_Factory}
gives a total sum of 
\[
(469/62) C\epsilon^{-1} \leq 7.57 C \epsilon^{-1}
\]
coin flips on average.
\end{proof}

\subsection{Small $r$}
\label{SEC:smallr}

Now that a recursive algorithm for large $r$ has been built, a
recursive analogue
for 
{\tt Small\_r\_1D\_Bernoulli\_Factory} works as follows.

\begin{center}
\setcounter{line}{1}
\begin{tabular}{rl}
\toprule
\multicolumn{2}{l}{{\tt Small\_r\_Bernoulli\_Factory} \quad {\em Input: } 
  $C,M$} \\
\midrule 
\gk & $\beta \leftarrow 1/(1-2M)$ \\
\gk & Draw $Y \leftarrow {\tt Logistic\_Bernoulli\_Factory}(\beta C)$ \\
\gk & Draw $B \leftarrow \berndist(1/\beta)$ \\
\gk & If $Y = 0$, then $X \leftarrow 0$ \\
\gk & Elseif $Y = 1$ and $B = 1$, then $X \leftarrow 1$ \\
\gk & Else $X \leftarrow {\tt Linear\_Bernoulli\_Factory}
  (C\beta/(\beta - 1))$ \\
\bottomrule
\end{tabular}
\end{center}

\begin{lemma}
Algorithm {\tt Small\_r\_Bernoulli\_Factory} produces a Bernoulli
distributed output with mean $Cp \leq M < 1/2$, 
and requires at most (on average)
\[
\frac{C}{(1 - 2M)(1 + Cp)}  + Cp\cdot \left[C 
  \frac{15.2}{1 - 2M + Cp}\right]
\]
coin flips to do so.
\end{lemma}

\begin{proof}
The proof is essentially the same as that of Lemma~\ref{LEM:Smallruntime}.
\end{proof}

\section{Lower bound}
\label{SEC:lower}

To see why it is unlikely that a method that uses fewer than 
$C\epsilon^{-1}$ coin flips can be constructed, consider building an
unbiased estimate of $p$.

The standard estimate is to generate $X_1,\ldots,X_n$ iid $\berndist(p)$,
and then use the sample average $\hat p_n = (X_1 + \cdots + X_n)/n$ as
an unbiased estimate of $p$.  This estimate is unbiased, and has 
variance $p(1 - p)/n$.

Now consider the estimate $Y/C$, where $Y \sim \berndist(Cp)$.  Then
$\mean[Y/C] = Cp/C = p$ so this estimate is also unbiased, and the 
variance is $Cp(1 - Cp)/C^2 = p(1 - Cp)/C.$  Therefore, this estimate
that used one draw from $\berndist(Cp)$ has the variance of 
the estimate that used 
$n = C(1 - p)/(1 - Cp)$ draws from the $p$-coin.

The Cram\'er-Rao lower bound (see, for instance~\cite{bickeld1977})
on the variance of an unbiased estimate of $p$
is 
\[
\frac{p(1 - p)}{n}.
\]
That is, any unbiased estimate that uses up to $n$ flips of the $p$-coin
must have variance at least $p(1-p)/n$.  That immediately gives that 
any algorithm for generating a $Cp$-coin that uses a deterministic number
$n$ of coin flips must have $n \geq C(1 - p)/(1 - Cp)$.  Of course,
this does not quite apply to a Bernoulli Factory, because here a random
number of coin flips is used.  

However, it is strong evidence that $C(1 - p)/(1 - Cp)$ is a lower bound
on the expected number of coin flips needed by an algorithm.

\section{Multivariate Bernoulli Factory}
\label{SEC:multivariate}

This new algorithm was designed for the single coin problem, in this
section consider generating a coin flip whose probability
of heads is the sum of the probability of heads on two different coins
each of which has an unknown probability of heads.  
Unlike the single coin flip, there
is no immediate application, however, it is useful to know that the
single coin algorithm can be easily generalized to solve this problem
should the need arise.

More generally, the goal is now to generate a coin flip with probability
\[
r = C_1 p_1 + \cdots + C_k p_k
\]
of heads, where $r$ is bounded away from 1, using as few flips of the coins
as possible.  When $k = 1$, this is the linear Bernoulli factory studied
in the previous sections.  Formally, a multivariate Bernoulli factory
is defined as follows.
\begin{definition}
Given a computable function $f:[0,1]^k \rightarrow [0,1]$, a
{\em multivariate Bernoulli factory} is a computable function 
\[
{\cal A}:[0,1] \times \left(\{0,1\}^{\{1,2,\ldots\}}\right)^k \rightarrow
 \{0,1\},
\]
such that if $U \sim \unifdist([0,1])$ and the $X_{i,j}$ are independent
random variables with 
$(\forall i \in \{1,\ldots,k\})(\forall j \in \{1,2,\ldots\})
 (X_{i,j} \sim X_i)$, then the following properties hold.
\begin{enumerate}
\item
There exist random variables 
$(T_1,\ldots,T_k) \in \{1,2,\ldots\}^k$ such that the value of 
${\cal A}(U,\{X_{1,i}\}_{i=1}^\infty,\ldots,\{X_{k,i}\}_{i=1}^\infty)$ only
depends on the values of $\{X_{1,i}\}_{i=1}^{T_1},\ldots,\{X_{k,i}\}_{i=1}^{T_k}$,
and for all $(t_1,\ldots,t_k)$, 
the event $(T_1,\ldots,T_k) = (t_1,\ldots,t_k)$ is measurable with respect
to $\{X_{1,i}\}_{i=1}^{t_1},\ldots,\{X_{k,i}\}_{i=1}^{t_k}$.
\item
${\cal A}(U,\{X_{1,i}\}_{i=1}^\infty,\ldots,\{X_{k,i}\}_{i=1}^\infty) \sim 
 \berndist(f(p_1,p_2,\ldots,p_k)).$
\end{enumerate}
Call $T_1 + \cdots + T_k$ the {\em running time} of the algorithm.
\end{definition}

The key to the single coin algorithm was generating a random
variable that was exponential with rate parameter $r$.  In the 
single coin case, this was done by generating a Poisson process of rate
$C$, then thinning.

For the multivariate coin case, consider generating $k$ independent
Poisson point processes $P_1,\ldots,P_k$, where $P_i$ has rate $C_i$.
Then thin each process $P_i$ with coin $i$ to obtain a Poisson point process
of rate $C_i p_i$.  The union of these is a new Poisson point process of
rate $r$, and the rest of the algorithm operates as before.  The proofs
of all the lemmas and Theorem~\ref{THM:second} then proceeds in 
exactly the same way as given earlier.

\section{Acknowledgments}

This research supported by NSF DMS-1418495.

\bibliographystyle{plainnat}

\end{document}